\documentclass[11pt]{amsart}
\usepackage{amsthm}
\usepackage{amsmath}
\usepackage{amsopn}
\usepackage{amstext}
\usepackage{bbm}
\usepackage{fullpage}
\usepackage{graphicx}
\usepackage{amssymb}
\usepackage{enumerate}
\newtheorem{theorem}{Theorem} %[section]

\newtheorem{corollary}[theorem]{Corollary}
\newtheorem{lemma}[theorem]{Lemma}
\newtheorem{claim}[theorem]{Claim}

\newtheorem{definition}[theorem]{Definition}
\def\eod{\vrule height 10pt width 9pt depth 0pt}

\newcommand{\la}{\lambda}
\newcommand{\ve}{\varepsilon}
\newcommand{\de}{\delta}
\newcommand{\De}{\Delta}
\newcommand{\al}{\alpha}

\newcommand{\be}{\beta}
\newcommand{\Om}{\Omega}
\newcommand{\Dom}{\partial\Omega}
\newcommand{\ga}{\gamma}
\newcommand{\Ga}{\Gamma}

\newcommand{\beq}{\begin{equation}}
\newcommand{\eeq}{\end{equation}}

\newcommand{\RR}{{\mathbb R}}
\newcommand{\ZZ}{{\mathbb Z}}

\newcommand{\EE}{{\mathbf E}}
\newcommand{\PP}{{\mathbf P}}

\newcommand{\ignore}[1]{{}}

\DeclareMathOperator{\dist}{dist}

\DeclareMathOperator{\supp}{supp}

\begin{document}


\title{The rate of convergence of the Walk on Spheres Algorithm}

\author{Ilia Binder\address{Ilia Binder, Dept. of Mathematics, University of Toronto.}}\thanks{The first author was partially supported by  NSERC Discovery grant 5810-2004-298433. This research was  partially conducted during the period the second
author  was employed by the Clay Mathematics Institute as a Liftoff Fellow.}
\author{Mark Braverman\address{Mark Braverman, Microsoft Research, New England}}
%\author{Mark Braverman\\mbraverm@cs.toronto.edu\\Dept. of Computer Science\\University of Toronto}

\date{}

\begin{abstract}
In this paper we examine the rate of convergence of one of the standard algorithms for emulating exit probabilities of Brownian motion, the Walk on Spheres (WoS) algorithm. We obtain the complete characterization of the rate of convergence of WoS in terms of the local geomnetry of a domain. 
 \end{abstract}

\maketitle

%\footnotetext[1]{Dept. of Mathematics, University of Toronto. }
%\footnotetext[2]{Microsoft Research. This research was  partially conducted during the period the
%author  was employed by the Clay Mathematics Institute as a Liftoff Fellow.}

%\newpage
%\setcounter{page}{1}
\section{Introduction}\label{intro}

The \emph{harmonic measure} on a bounded domain $\Om\subset\RR^d$ at $x\in\Om$ can be described as an exit distribution of Brownian motion (see \cite{garnet-marshall}). This measure plays an important role in various problems of Probability Theory, Geometric Function Theory, Dynamical Systems, Partial Differential Equations, as well as in a vast range of problems of Applied Mathematics. The problem of efficiently sampling from harmonic measure is therefore a key problem in Computational Mathematics.

One of the simplest and most commonly used  methods for sampling from harmonic measure is the \emph{Walk on Spheres (WoS)} algorithm. It was first proposed in 1956  by M. Muller in \cite{Muller}. Roughly speaking, the algorithm consists of replacing the Brownian Motion by a martingale $\{X_t\ :\ t\in\ZZ_{\geq0}\}$, such that $X_0=x$, and $X_t$ is uniformly distributed on a sphere centered at $X_{t-1}$ of a radius which is a certain proportion of the distance form $X_{t-1}$ to the boundary $\Dom$ (see Section \ref{sec:WoS} for the precise definition). 

It is not hard to see that it takes at most $O(1/\ve^2)$ steps for the WoS process to reach an $\ve$-neighborhood of $\Dom$. However, in many situations, this rate of convergence is unsatisfactory. In particular, if we wanted to get $2^{-n}$-close to the boundary, it would take us a number of steps {\em exponential} in $n$. As it turns out, that, depending on the local geometry of the boundary of the domain, the rate of convergence is {\em polynomial} or even {\em linear} in $n$ (i.e. logarithmic in $1/\ve$). 

Logarithmic rate of convergence  of the process $X_t$ to the boundary was established for convex domains by M. Motoo in \cite{Motoo}. It was later generalized by G.A. Mikhailov in \cite{Mikhailov} to planar domains satisfying any \emph{cone condition} (i.e. at every point of the boundary there is a cone of certain fixed opening in the complement of the domain), as well for  3-dimensional domains satisfying a cone condition with large enough surface angle. See also \cite{Mikhbook} and \cite{milstein} for additional historical background and the use of the algorithm for solving various types of boundary value problems. 

In our earlier work \cite{BB}, we established polylogarithmic, but not logarithmic, upper bounds on the rate of convergence of WoS for planar domains, and for a restricted class of higher-dimensional domains. Unfortunately, the techniques of \cite{BB} do not generalize well to general domains in higher dimensions.

 Our present results subsume all prior work on the rate of convergence of the WoS. 
We introduce an easily verified metric condition on the domain which provides tight bounds for the rates of convergence. We also show that the condition is tight.

\subsection{The Walk on Spheres algorithm}\label{sec:WoS}

Let us now define the WoS. We would like to simulate a BM in a given bounded domain $\Om$
until it gets $\ve$-close to the boundary $\partial \Om$. Of course one could simulate it using 
jumps of size $\de$ in a random direction on each step, but this would require $O(1/\de^2)$ steps. Since
we must take $\de = O(\ve)$, this would also mean that the process may take $O(1/\ve^2)$ steps to converge. 

The idea of the WoS algorithm is very simple:  we do not care about the path the BM takes, but only about 
the point at which it hits the boundary. Thus if we are currently at a point $X_t \in \Om$ and we know that 
$$
d(X_t):=d(X_t,\partial\Om)\ge r,
$$
i.e. that $X_t$ is at least $r$-away from the boundary, then we can just jump $r/2$ units in a random 
direction from $X_t$ to a point $X_{t+1}$. To justify the jump we observe that a BM hitting the boundary 
would have to cross the sphere 
$$S_t = \{ x~:~|x-X_t|=r/2\} $$
at some point, and the first crossing location $X_{t+1}$ is distributed uniformly on the sphere. There is nothing
special about a jump of $d(X_t)/2$ and it can be replaced with any $\be\, d(X_t)$ where $0<\be<1$.

Let $\{\ga_t\},\ t\in\ZZ_{\geq0}$ be
a sequence of i.i.d. random variables each being a vector uniformly distributed on the unit sphere in $\RR^d$. We could take,
for example, 
$\ga_t = {\Ga_t^d}/{|\Ga_t^d|}$,
where $\Ga_t^d$ is a normally distributed $d$-dimensional Gaussian variable. 
Then, schematically, 
the Walk on Spheres algorithm can be presented as follows:

\medskip
\begin{tabular}{|l|}
\hline
{\bf WalkOnSpheres}($X_0$, $\ve$)\\
$n := 0$;\\
{\bf while} $d(X_t) = d(X_t,\partial \Om)>\ve$ {\bf do}\\
$~~~~~$ compute $r_t$: a {\em multiplicative} estimate on $d(X_t)$ such that $\be\cdot d(X_t)< r_t< d(X_t)$;\\
$~~~~~$ $X_{t+1} := X_t + (r_t/2) \cdot \ga_t$;\\
$~~~~~$ $t := t+1$;\\
{\bf endwhile}\\
{\bf return} $X_t$\\
\hline
\end{tabular}
\medskip

Thus at each step of the algorithm we jump at least $\be/2$ and at most $1/2$-fraction of the distance to the 
boundary in a random direction. An example of running the WoS algorithm in $2$-d is illustrated on Figure \ref{fig:algexample}.

\begin{figure}[ht] 
\begin{center} \includegraphics[angle=0,scale=.8]{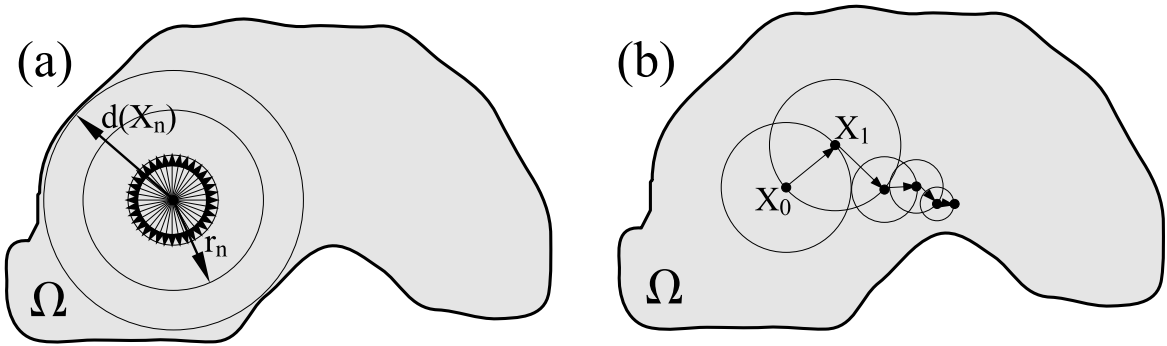}
\caption{An illustration of the WoS algorithm for $d=2$: one step jump (a), and a possible 
run of the algorithm for several steps (b)}
\label{fig:algexample}
\end{center}
\end{figure}

The proof of the convergence of the algorithm can be found, for example, in \cite{garnet-marshall}. Moreover, it is not hard to see that the WoS process hits the $\ve$-neighborhood of $\Dom$
in $O(1/\ve^2)$ steps. However, in many situations, this rate of convergence is unsatisfactory. In particular, if we wanted to
get $2^{-n}$-close to the boundary, it would take us a number of steps {\em exponential} in $n$. As it turns out, in many 
natural situations, the rate of convergence is {\em polynomial} or even {\em linear} in $n$ (i.e. logarithmic in $1/\ve$). The object of the paper is to prove that this is the case, and give precise condition on when the faster convergence occurs. 

While an actual implementation of the WoS would involve round-off errors introduced through an imperfect simulation,
 we will ignore those to simplify the presentation as they do not affect any of the main results.
Thus the problem becomes purely that of analyzing the family of stochastic processes $\{X_t\}$ and their convergence speed to $\Dom$.

\subsection{Results}
Let $H_\beta(K)$ denote the \emph{$\beta$-dimensional Hausdorff content of $K$}.
$$H_{\beta}(K)=\inf_{K\subset \cup B(x_j, r_j)}\sum r_j^{\beta}.$$ 

\begin{definition}
\label{def:cthick_main}
A domain $\Om \subset \RR^d$ is said to be $\alpha$-thick  $0\leq\alpha\leq d$ if there exists a constant $C > 0$
such that for every $x \in \partial \Om$ 
$$H^{d-\al}(B(x,r)\setminus \Omega)\geq Cr^{d-\al},\qquad r<1$$
\end{definition}

Roughly speaking, $\al$-thick domains  have complements of codimension $\al$, which are uniformly large at every scale at every boundary point.

We call the constant $c$ \emph{the thickness of the domain $\Omega$}. It is not hard to see that the 
property of $\al$-thickness is monotone: an $\al$-thick domain is $\al'$-thick for $\al<\al'\le d$. 

Let us list some examples of $\al$-thick domains. 

\begin{enumerate}
\item
All $d$-dimensional domains are $d$-thick; 
\item
all bounded $d$-dimensional domains $\Om$ such that the 
complement $\Om^c$ is connected are $d-1$-thick.
\item all convex domains and all domains satisfying cone condition are $0$-thick;
\item all domains $\Om$ that are bounded by a smooth hypersurface $\Dom$ are $0$-thick. 
\end{enumerate}

It turns out that the $\al$-thickness of the domain is responsible for the rate of convergence of the WoS algorithm. This idea is formulated precisely in our Main Theorem.

\begin{theorem}
\label{thm:main}
Let $\Om$ be a bounded $\al$-thick domain in $\RR^d$. Then the expected rate of convergence of the WoS from 
any $x\in \Om$ until termination at distance $<\ve$ to the boundary is given by the following table:
\begin{equation}
\label{table:main}
\text{
\begin{tabular}{|l|c|}
\hline
 & Rate of convergence \\
\hline \hline 
$\al<2$ & $O\bigl(\log 1/\ve\bigr)$ \\
\hline
$\al=2$ & $O\bigl(\log^2 1/\ve\bigr)$ \\
\hline
$\al>2$ & $O\bigl((1/\ve)^{2-4/\al}\bigr)$ \\
\hline
\end{tabular}}
\end{equation}
The $O(\cdot)$ in the expressions above depends on the dimension $d$, on $\al$, on the thickness constant $C$ from Definition 
\ref{def:cthick_main} and on $\be>0$ from the definition of the WoS. It does not depend directly on $\Om$. 

Moreover, the rates of convergence above are tight. That is, for each $\al$ there is a family of $\al$-thick domains  $\Om^\al_n$ 
with some thickness $C$, such that the rate of convergence with $\ve=1/n$ on $\Om^\al_n$ is asymptotically given by the formulas in 
\eqref{table:main}. 
\end{theorem}

The rate of convergence cannot be better than $O(\log 1/\ve)$ since at each step of the WoS, the distance of $X_t$ to 
the boundary $\Dom$ decreases by at most a multiplicative constant. An intuitive explanation to the phase transition phenomenon occurring 
at $\al=2$, is that a BM in $\RR^d$ almost surely ``misses" sets of co-dimension $>2$, while hitting 
sets of co-dimension $<2$ with positive probability. 

It is worth noting that the main result in \cite{BB} is the special case $\al=d=2$ of the theorem.

The following corollaries are implied directly by the Theorem \ref{thm:main}. 

\begin{corollary}
\label{cor:main}
\begin{enumerate}
\item 
Since any planar domain is $2$-thick, the WoS converges in $O(\log^2 1/\ve)$ steps;
\item 
since any planar domain with connected exterior is $1$-thick, the WoS converges in $O(\log 1/\ve)$ steps;
\item 
since any domain in $\RR^d$ is $d$-thick, for $d\ge 3$ the WoS converges in $O((1/\ve)^{2-4/d})$ steps;
\item 
since any $3$-dimensional domain with connected exterior is $2$-thick, the WoS converges in $O(\log^2 1/\ve)$ steps;
\item 
since for  any $d\ge 4$, any $d$-dimensional domain with connected exterior is $d-1$-thick, the WoS converges in $O((1/\ve)^{2-4/(d-1)})$ steps;
\item 
since any domain bounded by a smooth hypersurface is $0$-thick, the WoS converges in $O(\log 1/\ve)$ steps.
\end{enumerate}
\end{corollary}

The rest of the paper is organized as follows. In Section \ref{sec:upper} we construct the auxiliary boundary barrier measures and the energy functions.   Using these functions, we prove the upper estimates of Theorem \ref{thm:main}. More technical estimates on the energy function are done in Section \ref{sec:estimates}. Finally, in Section \ref{sec:lower}, we present examples of $\al$-thick domains with the slow rate of convergence of the WoS process.

\section{Upper bounds: energy functions}\label{sec:upper}

\subsection{Construction of an auxiliary measure}

In this section we will construct a family of measures near boundary points of an $\al$-thick domain. These measures will be used to construct energy functions, which, in turn, play crucial role in the proof of  Theorem \ref{thm:main}.

\begin{lemma}\label{lem:cthick}
There exists a constant $c=c(\al, d, C)$ such that for any  $\al$-thick domain $\Om$ with thickness $C$ in $\RR^d$ and for any $x\in\Dom$, one can find a Borel measure $\mu_{x}$ which satisfies the following conditions:
\begin{enumerate}
\item 
$\supp(\mu_x)\cap \Om = \emptyset$, or, equivalently, $\mu_x(\Om)=0$; 
\item 
for any $y\in \RR^d$ and $r>0$, $\mu_x(B(y,r))\le r^{d-\alpha}$;
\item 
for any $r<1$, $\mu_x (B(x,r))\ge c \cdot r^{d-\alpha}$.
\end{enumerate} 
\end{lemma}
With a slight abuse of notation, we will also refer to the constant $c$ from this lemma as the \emph{thickness} of the domain.
\begin{proof}
The proof of the Lemma follows the standard reasoning that can be found in, say, Chapter II of \cite{carbook}.

Let us consider the dyadic grid selected so that the point $x$ has coordinates $(1/3, 1/3,\cdots,1/3)$. For an integer $d$-multi-index $\gamma=(\gamma_1,\dots, \gamma_d)$, let $D_{k,\gamma}$ be the cube $$\{(x_1,\dots, x_d\ :\ \gamma_n 2^{-k} \leq x_n<(\gamma_n+1 )2^{-k},\ n=1,\dots, d\}.$$ Let $D_k(x)$ be the unique dyadic cube of the size $2^{-k}$  which contains $x$. Note that $x$ is always at   distance $2^{-k}/3$ from the boundary of $D_k(x)$.

We will construct inductively the sequence of measures $\nu_n$. They will satisfy the following properties:
\begin{enumerate}[(a)]
  \item $\supp\nu_n\cap\Om=\emptyset$.
  \item $\nu_n(D_{k,\gamma})\leq H^{d-\al}(D_{k,\gamma}\setminus\Omega)$ for $1\leq k\leq n$
  \item \label{cond:c} $\nu_n(D_{k}(x))= H^{d-\alpha}(D_{k}(x)\setminus\Omega)$ for $1\leq k\leq n$
\end{enumerate}

Let $\nu_1$ be a delta measure in a point of  $D_1(x)\setminus\Omega$ with the total mass $H^{d-\al}(D_1(x))$. It clearly satisfies all of our assumptions. 

Assume now that the measure $\nu_n$ has already been constructed. The measure $\nu_{n+1}$ will be a sum of delta-measures on the points outside of $\Omega$ lying in the cubes from the $n+1$-st dyadic generation, such that $\nu_{n+1}(D_{n,\gamma})=\nu_{n}(D_{n,\gamma})$ for all $\gamma$ (so $\nu_{n+1}$ will be obtained from $\nu_n$ by re-distributing the latter over the cubes of the $(n+1)$-st generation). Thus the measure $\nu_{n+1}$ would automatically satisfy the second and the third condition for $k\leq n$.

To construct $\nu_{n+1}$, we use the following rule. 

First, we set $\nu_{n+1}(D_{n+1}(x))=H^{d-\al}(D_{n+1}(x)\setminus\Omega)$. The measure $\nu_{n+1}$ clearly satisfies our condition (\ref{cond:c}) on $D_{n+1}(x)$.

Second, for any other dyadic cubes $D_{n+1, \gamma}\subset D_n(x)$, we assign the mass
\begin{equation}\nu_{n+1}(D_{n+1, \gamma})=H^{d-\al}(D_{n+1, \gamma}\setminus\Omega)\frac {(\nu_n(D_n(x))-\nu_{n+1}(D_{n+1}(x)))}{\sum_{D_{n+1,\delta}\subset D_n(x),\ D_{n+1,\delta}\neq D_{n+1}(x)}H^{d-\al}(D_{n+1, \delta}\setminus\Om)},
\end{equation}
so that $\nu_n(D_n(x))=\nu_{n+1}(D_n(x))$. By sub-additivity of the Hausdorff content,
$$\sum_{D_{n+1,\delta}\subset D_n(x)}H^{d-\al}(D_{n+1, \delta}\setminus\Om))\geq H^{d-\al}(D_{n}(x)),$$
and hence \begin{equation}\label{eq:second}\nu_{n+1}(D_{n+1, \gamma})\leq H^{d-\al}(D_{n+1, \gamma}\setminus\Omega)\end{equation} for  $D_{n+1,\gamma}\subset D_n(x)$. 

Finally, for any other dyadic cubes from $(n+1)$-st generation, we set
\begin{equation}\nu_{n+1}(D_{n+1, \gamma})=H^{d-\al}(D_{n+1, \gamma}\setminus\Omega)\frac {\nu_n(D_n)}{\sum_{D_{n+1,\delta}\subset D_n}H^{d-\al}(D_{n+1, \delta}\setminus\Om)},\end{equation}
where $D_n$ is the unique cube from the $n$-th dyadic generation containing $D_{n+1, \gamma}$. 
Using the sub-additivity of the Hausdorff content, as above, we get the estimate \eqref{eq:second} for all   cubes of the $(n+1)$-st generation. 
The construction again satisfies $\nu_{n+1}(D_n)=\nu_{n}(D_n)$. 

Let now $\nu$ be any weak$^*$ limit point of the sequence $\nu_n$. $\nu$ is still supported outside of $\Om$. By the second property of the measures $\nu_n$, 
\begin{equation}
\label{eq:nu1fin}
\nu(D_{k,\gamma})\leq \sum_{\overline{D_{k,\delta}}\cap \overline{D_{k,\gamma}}\neq\emptyset} H^{d-\al}(D_{k,\delta}\setminus\Omega)
\le 3^d\, H^{d-\al}(D_{k,\gamma}) < 3^d\, (\sqrt{d})^d \, 2^{-k(d-\al)}\end{equation} for all $k$.
Using the third property of the measures $\nu_n$, the $\al$-thickness of $\Omega$, and the fact that $D_k(x)$ contains the ball of the radius $2^{-k}/3$, we get that 
\begin{equation}\label{eq:nu2fin}\nu(D_{k}(x))\geq  H^{d-\al}(D_{k}(x)\setminus\Omega)\geq c \, 3^{-d}\,2^{-k(d-\al)}\end{equation} for any $k$.

Every ball can be covered by certain ($d$-dependent) number of dyadic cubes of comparable size, so \eqref{eq:nu1fin} implies that $\nu(B(y,r)) \lesssim r^{d-\alpha}$. Every ball centered at $x$ also contains  a dyadic cube of comparable (again, $d$-dependent) size, hence by \eqref{eq:nu2fin},   $\nu(B(x,r)) \gtrsim r^{d-\alpha}$. Now we can set $\mu_x$ to be an appropriately normalized measure $\nu$.

\end{proof}

\subsection{ Energy Function of optimal growth}\label{subsec:energy}
The heart of the proof of the upper bounds in Theorem \ref{thm:main} is the construction of a subharmonic function with optimal growth at the boundary, the \emph{Energy Function} $U$ on $\Om$. We will construct $U(x)$ so that it is ``small" in the interior of $\Om$, and grows to $\infty$ as $x$ approaches the boundary $\Dom$. The $\alpha$-thickness of the domain allows us to establish that the value of $U(X_t)$ grows in expectation 
as the WoS progresses. Thus after a certain number of steps $U(X_t)$ will be large in expectation which would imply that $X_t$ is close 
to $\Dom$ with high probability.

The construction of the function is based on the notion of a \emph{Riesz potential}.
For a finite Borel measure $\mu$ on ${\mathbb R}^d$, and $\alpha<d$, the {\it $\alpha$-Riesz potential of the measure $\mu$} is defined by  
$$U_{\alpha}^{\mu}(x)=\frac1{d-\alpha}\int\frac{d\mu(z)}{|z-x|^{d-\alpha}}.$$
For $\alpha=d$, the $d$-Riesz potential is defined by
$$U_{\alpha}^{\mu}(x)=\int\log\frac1{|z-x|}\,d\mu(z).$$
The value $U_{\alpha}^\mu(x) = \infty$ is allowed when the integral diverges.

An important special case is the case of  $\alpha=2$, the so-called {\emph Newton potential}. We will denote $U_2^{\mu}$ simply by $U^{\mu}$.
In this case the expression under the integral is harmonic in $\RR^d$. 
It is well known (e.g. see \cite{Landkoff}) that the function $U^{\mu}$ is superharmonic on ${\mathbb R}^d$, and harmonic outside of $\supp \mu$. 

More generally, outside of the $\supp \mu$, we have the identity
\begin{equation}\label{eq:laplace}
\Delta U_{\alpha}^{\mu}(y)= (d-\alpha+2)(2-\alpha)U_{\alpha-2}^{\mu}(y).
\end{equation}
It shows that for $0<\alpha<2$, the function $U^{\mu}_{\alpha}$ is subharmonic outside of $\supp\mu$.

The following important technical identity, which easily follows from Fubini's Theorem and substitution,  relates the local behavior of the measure $\mu$ and the growth of its potential $U_{\alpha}^{\mu}$.
For $\alpha<d$, we have
\begin{equation}\label{eq:rewrite}
U_{\alpha}^{\mu}(y)=\frac1{d-\alpha}\int_0^{\infty}\mu(B(y, t^{-1/(d-\alpha)}))\,dt=\int_0^{\infty}\frac{\mu(B(y, r))}{r^{d-\alpha+1}}\,dr,
\end{equation}
 and for $\alpha=d$,
\begin{equation}\label{eq:rewrite2}
U_{\alpha}^{\mu}(y)=\int_{-\infty}^{\infty}\mu(B(y, e^{-t}))\,dt=\int_0^{\infty}\frac{\mu(B(y,r))}r\,dr
\end{equation}

Let us now fix an $\alpha$-thick domain $\Omega\subset B(0,1)\subset \RR^d$. Let us consider the set $\mathcal M$ of all Borel measures $\mu$ supported inside $\overline{B(0,2)}$ and outside of $\Omega$ (i.e. $\mu(\Omega)=0$), satisfying the following condition:
\begin{equation}\label{eq:measure}
\text{for any } y\in \RR^d\text{ and }r>0,\ \mu(B(y,r))\le r^{d-\alpha}
\end{equation}

Let us now introduce the {\em Energy Function} $U(y)$. Recall that $U^\mu(y):=U_2^\mu(y)$. 
\begin{equation}\label{eq:energy}
U(y) := \left\{
\begin{array}{ll}
\sup_{\mu\in\mathcal M}U_{\alpha}^{\mu}(y), & \text{when } \alpha\leq 2 
\\ \sup_{\mu\in\mathcal M}U^{\mu}(y), &  \text{when }\alpha\geq2.
\end{array}
\right.
\end{equation}

Since the set $\mathcal M$ is weakly$^*$-compact, for every $y\in\Dom$ there exists a measure maximizing the potential in \eqref{eq:energy} at the point $y$.

Let us summarize the properties of $U(y)$  in the following claim. The proof uses the
 identities \eqref{eq:rewrite} and \eqref{eq:rewrite2}. Recall that $d(y)=\dist(y,\partial\Omega)$.

\begin{claim}\label{lem:energy}
Let $\Om$ be an $\alpha$-thick domain. Then
\begin{enumerate}
\item $U(y)$ is subharmonic in $\Omega$. 
\item For $\alpha\leq 2$, 
$\displaystyle{U(y) \le \log \frac{2}{d(y)}}$
for all $y\in \Omega$.
\item
For $\alpha> 2$, 
$\displaystyle{U(y) \le \frac1{\alpha-2}d(y)^{2-\alpha}}$
for all $y\in \Omega$.
\end{enumerate}
\end{claim}

\begin{proof}
Let $\alpha\leq 2$, $y\in\Omega$ and $\mu\in\mathcal M$. Equations \eqref{eq:rewrite} and \eqref{eq:rewrite2} imply that 
\begin{equation}\label{eq:upper}
U^{\mu}_{\alpha}(y)\leq\int_{d(y)}^2\frac1t\,dt=\log\frac2{d(y)}.
\end{equation}
Similarly, for $\alpha>2$, we will use the harmonic potential $U^{\mu}$.
Let $\al>2$, $y\in\Omega$ and $\mu\in\mathcal M$. Once again,
 \eqref{eq:measure}, $\supp\mu\cap\Omega=\emptyset$, and the equations \eqref{eq:rewrite} imply that
\begin{equation}\label{eq:upper_slow}
U^{\mu}(y)\leq\int_{d(y)}^2\frac1{t^{\alpha-1}}\,dt\leq \frac1{(\alpha-2)d(y)^{\alpha-2}}.
\end{equation}
By equations \eqref{eq:upper} and \eqref{eq:upper_slow}, $U(y)$ is a supremum of a locally bounded family
of subharmonic functions. Thus $U(y)$ is subharmonic. 

The second and third statements of the claim follow directly from \eqref{eq:upper} and \eqref{eq:upper_slow} respectively.
\end{proof}

Let $X_t$ be the WoS process initiated at some point $X_0 =y\in\Omega$. Let us define a new process $U_t=U(X_t)$, the value of the energy function at the $t$-th step of the process. Note that because $U$ is subharmonic , $U_t$ is a submartingale, that is $\EE[U_{t+1}|U_t]\ge U_t$.

For the rest of the section let $n=1/\ve$. 
Claim \ref{lem:energy} immediately implies that a large value of $U_t$ will guarantee the closeness to the boundary. More specifically,
\begin{claim}\label{lem:dist}
For $\al\le 2$,  if $U_t>\log 2n$ then $d(X_t)<1/n$.

For $\alpha>2$, $U_t>(\alpha-2) n^{\alpha-2}$ implies $d(X_t)<1/n$.
\end{claim}

The proof of Theorem \ref{thm:main} relies on   finer lower bounds on the function $U$, which would guarantee the optimal rate of boundary convergence. We prove the bounds in the next section. These bounds depend heavily on the value of $\alpha$. We first give a probabilistic
proof of the upper bounds in Theorem \ref{thm:main}, and then prove the finer estimates on $U$ in Section \ref{sec:estimates}.

\subsection{Logarithmic convergence: the case $\alpha<2$.}\label{subsec:al_small}

In the heart of the proof for this case lies the following strong estimate on the behavior of the Riesz potentials near the boundary.
\begin{lemma}\label{lem:energy_upper_small} 
For any $\al<2$ and $c>0$, there exist
two constants $\delta$ and $\eta$, such that the following holds.

Let $\Omega$ be an $\alpha$-thick domain in $\RR^d$ with thickness $c$. Let $y\in\Omega$ and $x\in\partial\Omega$ be the closest point to $y$. Let $\mu\in\mathcal M$.

Then either  
\begin{equation}
\label{eq:nu1}
U(z)>U_{\al}^{\mu}(z)+1 \text{ whenever } \delta/4\cdot d(y)<|z-x|<\delta \cdot d(y).
\end{equation}
or
\begin{equation}\label{eq:nu2}
\mu(B(y,2d(y)))\geq\eta d(y)^{d-\alpha}
\end{equation}
\end{lemma}
The lemma is established in Section \ref{sec:growth}.

Note that after $k=O(|\log \delta|)$ steps of the WoS process,
\begin{equation}\label{eq:dst} \delta/4\cdot d(X_t)<|X_{t+k}-x|<\delta \cdot d(X_t)\text{ with a certain probability $p$},\end{equation}
where $x$ is the point of $\Dom$ that is closest to $X_t$,  and $p>0$ depends only on $\beta$ and the dimension $d$. 

Let us fix $X_t$ and take the measure $\mu\in\mathcal M$ maximizing the value of $U^{\mu}_\al(X_t)$.
By the preceding observation, in the first case in Lemma \ref{lem:energy_upper_small}, the subharmonicity of $U$ implies that the expectation of $U_{t+k}$, conditioned on $U_t$, will increase by some definite constant.
   
On the other hand, using the identity \eqref{eq:laplace} and the $\al$-thickness of $\Om$, one can see that the Laplacian of $U^\mu_\al$ is large near the point $X_t$ in the second case of Lemma \ref{lem:energy_upper_small}. Thus, since large Laplacian leads to a fast build-up of mean values, we have the above-mentioned increase by a constant after the 
first step. We arrive at the following estimate, which shows that $U_t$ grows at least linearly in 
expectation. 

\begin{lemma}\label{lem:upward}
There are  constants $L$ and $k$, depending only on $c$, $\be$, and $\alpha$,  such that $$\EE[(U_{t+k}-U_t)|U_t]>L.$$
\end{lemma}

A detailed proof of the lemma can be found in  Section \ref{sec:drift}.

Lemma \ref{lem:upward} implies that $\EE[U_t]>tL/k+U_0$. Since $d(X_t)\geq(1-\beta)^td(X_0)$, Claim \ref{lem:dist} implies that $U_t\leq U_0+t|\log(1-\beta)|+\log2$. This implies that $U_t>U_0 + tL/2 k$ with probability at least $P$, where $P$ depends only on $\beta$. This, together with Claim \ref{lem:dist} implies the necessary upper bound in the case $\al<2$.

\subsection{Polylogarithmic convergence: the case $\alpha=2$.}\label{subsec:al_2}
In the case $\al=2$ the steady linear growth of $U_t$ given by Lemma \ref{lem:upward} no longer
holds. In fact, the only thing that generally holds in this case is the submartingale property
$\EE[U_{t+1}|U_t]\ge U_t$. We are able to overcome this difficulty by showing that the submartingale 
process $\{U_t\}$ has a deviation bounded from below by a constant at every step. To this end it 
suffices to show that $U_t$ can grow by some $\eta$ with a non-negligible probability. 
We use the following estimate on the energy function (established in Section \ref{sec:growth}).
\begin{lemma}\label{lem:energy_upper_2} 
There exists
a constant $\delta$, dependent only on the thickness $c$, such that the following holds.
Let $\Omega$ be a $2$-thick domain. Let $y\in\Omega$ and $x\in\partial\Omega$ be the closest point to $y$. Then 
\begin{equation}
\label{eq:nu}
U(z)>U(y)+1 \text{ whenever } |z-x|<\delta \cdot d(y).
\end{equation}
\end{lemma} 

Since the function $U$ is subharmonic,   observation \eqref{eq:dst} implies the following estimate (see Section \ref{sec:drift} for a  proof).
\begin{lemma}\label{lem:main}
Let $\Omega$ be a $2$-thick domain in $\RR^d$.
There are  constants $k$ and $L$, depending only on the thickness $c$, the jump ratio $\be$, and the dimension $d$, such that $$\EE[(U_{t+k}-U_t)^2|U_t]>L.$$
\end{lemma}

We can now use Lemma \ref{lem:main} to prove the upper bounds on the rate of convergence for the case $\al=2$.
Let us replace the submartingale $U_t$ by a stopped submartingale 
$$V_t=\begin{cases}U_t,&t<T_n\\
U_{T_n}, &t\geq T_n
\end{cases}$$
By the optional stopping time theorem (see \cite{KS}), $V_t$ is also a positive submartingale; $V_t\leq\log\frac{4}{n}$.
This implies, in particular, that
\begin{equation}\label{eq:jump}
\EE[V_{t}(V_{t+k}-V_{t})]=\EE[\EE[V_{t}(V_{t+k}-V_{t})|V_{t}]]\geq\EE[\EE[V_{t}(V_{t}-V_{t})|V_t]]=0
\end{equation}

Lemma \ref{lem:main} implies that
\begin{equation}\label{eq:cor}
\EE[(V_{t+k}-V_t)^2]>L\cdot\PP[T_n>t+k].
\end{equation} 

We are now in  a position to establish the upper bounds for $\al=2$.
\begin{proof}[{\bf Proof of the upper bound from Theorem \ref{thm:main} for $\al=2$.}]
Assume first that for some M, $${\bf P}[T_n > M \log^2 n]\geq1/2.$$ It means that for all 
$t\leq M \log^2 n-k$, $\PP[T_n\geq t+k]\geq1/2$. 
This implies
\begin{multline*}
\EE[V_{t+k}^2]=\EE[((V_{t+k}-V_{t})+V_{t})^2]=\EE[V_{t}^2]+\EE[(V_{t+k}-V_{t})^2]+2\EE[V_{t}(V_{t+k}-V_{t})]\geq \EE[V_{t}^2]+L/2.
\end{multline*}
The last inequality follows from  \eqref{eq:jump} and \eqref{eq:cor}.
Hence $\displaystyle{\EE\left[V_{M \log^2 n}^2\right]\geq\frac{LM \log^2 n}{2k}}$.
Since $V_t\leq\log\frac{4}{n}$, this leads to a contradiction for large enough $M$.
\end{proof}

\subsection{Polynomial convergence: the case $\al>2$.}\label{subsec:al_large}

For the case $\al>2$, the required converse to Claim \ref{lem:dist} is relatively simple.
\begin{lemma}\label{lem:energy_low_large}
For $\alpha> 2$, and an $\alpha$-thick domain $\Omega$ in $\RR^d$ with the thickness $c$,
$$U(y) \ge K \cdot d(y)^{2-\alpha}$$
for all $y\in \Omega$. Here the constant $K=K(c,\alpha)$ depends only on $c$ and $\alpha$.
\end{lemma}
The lemma is established in  Section \ref{sec:growth}.

The idea of the proof of Theorem \ref{thm:main} in this case is now as follows. When the WoS is far
from the boundary $\Dom$ it makes fairly big steps and when it is close it makes small steps. 
There are not too many big steps because the number of big steps of length $>\ve$ confined
to $B(0,1)$ is bounded by $O(1/\ve^2)$. On the other hand, there are not too many small 
steps, because a small step means that the WoS is very close to $\Dom$, and should converge 
before  an opportunity to make many more steps.

More precisely, the number of ``big'' jumps is bounded by the following Claim.
\begin{claim}\label{clm:bigsteps}
Let $N(\ve, T)$ be the number of the jumps in the WoS process before the time $t$ which are bigger then $\ve$, i.e. $$N(\ve,T)=\# \{t\leq T\ |\ |X_t-X_{t-1}|\geq \ve\}.$$
Then 
$$\PP\left[N(\ve,T)>\frac4{\ve^2}\right]<1/4.$$
\end{claim}

\begin{proof}
Note now that because $X_t$ is a martingale, we have
\begin{multline}
1\geq \EE[X_T^2]-X_0^2=\sum_{k=1}^T \left(\EE[X_t^2]-\EE[X_{t-1}^2]\right)=\sum_{t=1}^T \EE[\left(X_t-X_{t-1}\right)^2]=\\ \EE[\sum_{t=1}^T \left(X_t-X_{t-1}\right)^2]\geq \ve^2 \EE[N(\ve, T)]
\end{multline}

The last equation implies the statement of the claim, by Tschebyshev inequality. \end{proof}

\begin{figure}[ht] 
\begin{center} \includegraphics[angle=0,scale=.65]{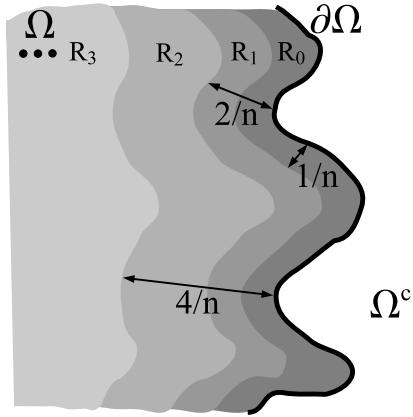}
\caption{The regions $R_k$ from the proof of the upper bounds for $\al>2$}
\label{fig:ubfigB}
\end{center}
\end{figure}

\ignore{
\begin{figure}[ht] 
\begin{center} \includegraphics[angle=0,scale=.65]{ubfigs.eps}
\caption{(a) Construction of the amalgamation $\nu$; (b) The regions $R_k$ from the proof of the upper bounds for $\al>2$}
\label{fig:ubfig}
\end{center}
\end{figure}
}

To bound the number of small jumps, we 
denote by $R_0\subset \Om$ the $1/n$-neighborhood of $\Dom$, and more generally, by 
$$R_k := \{x \in \Om ~:~2^{k-1}/n<d(x,\Dom)\le 2^{k}/n\}$$
(see Figure \ref{fig:ubfigB}).  Note that by Lemma \ref{lem:energy_low_large}, we have
\begin{equation}\label{eq:Rk}
2^{(k-1)(2-\alpha)}n^{\alpha-2}\geq U(y)\geq K 2^{k(2-\alpha)}n^{\alpha-2}
\end{equation}
for $y\in R_k$. Using this fact we prove the following. 

\begin{claim}
\label{clm:smallsteps}
Denote by $v_k$ the number of visits of $X_t$ to $R_k$ before the time $T$ when $X_t$ first hits the $1/n$-neighborhood of the boundary $\Dom$,
$$ v_k = \#\{ t < T~:~X_t\in R_k\}. $$
Then $$\PP[v_k > C_2 \cdot 2^{k (\al-2)} M] < 1/4^{M},$$ for some constant $C_2 =C_2 (c,d,\al,\be)$ and for any $M>1$.
\end{claim}

\begin{proof}
Suppose that at some 
point $t$, $X_t \in R_k$. We estimate from below the probability that this is the {\em last} time the WoS visits $R_k$. 

First of all, with some probability $p=p(\al,\be,c,)>0$ and for some constant $\eta$, $X_{t+\eta}\in R_{k+2\log c}$, i.e. the series of first jump brings us much closer to $\Dom$.
 Consider the subharmonic function 
$$
\Phi(y) = {(2 n)^{2-\alpha}}U(y)-2^{k(2-\alpha)}.
$$
in $\Omega$. Then the process $\Phi(X_{t+j})$ is a submartingale. We stop it at time $t+\tau$, $\tau\geq\eta$, when either the WoS terminates or when 
$X_{t+\tau}\in R_k$ (i.e. the process gets back to $R_k$), whichever comes first. 
If $X_{t+\tau}$ is $1/n$-close to $\Dom$ (but not closer than $1/2n$), then $\Phi(X_{t+\tau})\le 1$, by \eqref{eq:Rk}. 
If $X_{t+\tau} \in R_k$, then, again by \eqref{eq:Rk}, $\Phi(X_{t+\tau}) \le 0$. Another application of \eqref{eq:Rk} implies  that 
if $y\in R_{k+2\log c}$, then $\Phi(y)>\ga$ for some constant $\ga$.
Thus the probability that the WoS terminates at $X_{t+\tau}$  (i.e. we
never visit $R_k$ again) is at least
$$
\PP[X_{t+\tau}\notin R_k] \ge \EE[\Phi(X_{t+\tau})] \ge \Phi(X_{t+\eta}) \ge p\ga.
$$
Thus the probability that the visit $X_t$ to $R_k$ is the last one is at least $p\cdot \ga$. 
The claim now follows from an estimate of the probability of having at least $v_k$ returns to $R_k$, each of them {\em not} being the last one. 
\end{proof}

Claims \ref{clm:bigsteps} and \ref{clm:smallsteps} together imply  the upper bounds on the rate of convergence for $\al>2$.

\begin{proof}[{\bf Proof of the upper bounds from Theorem \ref{thm:main} for $\al>2$.}]
By  Claim \ref{clm:smallsteps}, for any $k$, we have that 
$$\PP[v_{k-s} > C_2 \cdot 2^{(k-s) (\alpha-2)}\cdot (3/2+s/2)] < 1/4^{3/2+s/2} = (1/8)\cdot 2^{-s}.$$
Hence, by union bound $v_{k-s} \le C_2 \cdot 2^{(k-s) (\alpha-2)}\cdot (3/2+s/2)$ for
all $s\ge 0$  with probability at least $3/4$. Let $k$ be such that $2^k \approx n^{2/\alpha}$. Then, with the  
probability at least $3/4$, we have the total number of jumps smaller than $2^k/n$ bounded by
\begin{equation}
\label{eq:smallsteps}
\sum_{s=0}^{ k} v_{k-s} \le \sum_{s=0}^{k} C_2 \cdot 2^{(k-s) (\al-2)}\cdot (3/2+s/2) < 
4 C_2 \cdot 2^{k(\al-2)} \approx 4 C_2 \cdot n^{2-4/\al}.
\end{equation}
If we take $N= (C_1 + 8 C_2) n^{2 - 4/d}$ steps of the WoS, 
\eqref{eq:smallsteps} implies that at least half the steps would be of magnitude at least 
$2^k/n \approx n^{2/d - 1}$, except with probability $<1/4$. Applying the estimate from Claim \ref{clm:bigsteps}  with $\ve = 2^k/n$, we see 
that   with the probability at least $3/4$,  $N(2^k/n,t)\leq 4n^2/2^{2k}\approx 4 n^{2-4/\al}$. Hence with probability $\ge 1/2$ the WoS terminates after
$O(n^{2-4/\al})$ steps.
\end{proof}

\section{Boundary behavior of the energy function}
\label{sec:estimates}

In this section we prove the analytical estimates on the behavior of the energy function 
that have been used in Section \ref{sec:upper}.

\subsection{Estimating boundary growth}
\label{sec:growth}

%We begin with lower bounds on the energy, stated in Sections 
%\ref{subsec:al_small}, \ref{subsec:al_2} and \ref{subsec:al_large}.
We start with the easiest case $\al>2$.

\noindent
{\bf Lemma \ref{lem:energy_low_large} (Section \ref{subsec:al_large}):}
For $\alpha> 2$, and an $\alpha$-thick domain $\Omega$ in $\RR^d$ with the thickness $c$,
$$U(y) \ge K \cdot d(y)^{2-\alpha}$$
for all $y\in \Omega$. Here the constant $K=K(c,\alpha)$ depends only on $c$ and $\alpha$.
\smallskip

\begin{proof}
Let $x$ be the closest to $y$ point at $\partial\Omega$, and let $\mu=\mu_x$ be the corresponding measure from the definition of the $\alpha$-thick domains.
Then, by the identity \eqref{eq:rewrite} and since $B(x,r)\subset B(y, r+d(y))$,
\begin{multline}
U(y)\geq U_{2}^{\mu}(y)=\int_{d(y)}^{2}\frac{\mu(B(y, r))}{r^{d-1}}\,dr\geq
\int_{2d(y)}^{2}\frac{\mu(B(x, r-d(x)))}{r^{d-1}}\,dr\geq\\ \int_{2d(y)}^{2}\frac{\mu(B(x, r/2))}{r^{d-1}}\geq c 2^{\alpha-d} \int_{2d(y)}^{2}t^{1-\alpha}\geq K\cdot d(y)^{2-\alpha}.
\end{multline}
\end{proof}

Unfortunately, in the case $\alpha\leq 2$   lower bounds of the type established  in the proof of Lemma \ref{lem:energy_low_large} are insufficient, and we will use   finer estimates provided by the following construction.  

Let $y$ be a point in $\Omega$, $x$ be the point of $\partial\Omega$ that is the closest to $y$, and $\mu$ be a measure in the class $\mathcal M$. We construct a new measure $\nu\in\mathcal M$, which we call the \emph{amalgamation} of $\mu$ at the point $y$ in the following way.

\begin{figure}[ht] 
\begin{center} \includegraphics[angle=0,width=\textwidth]{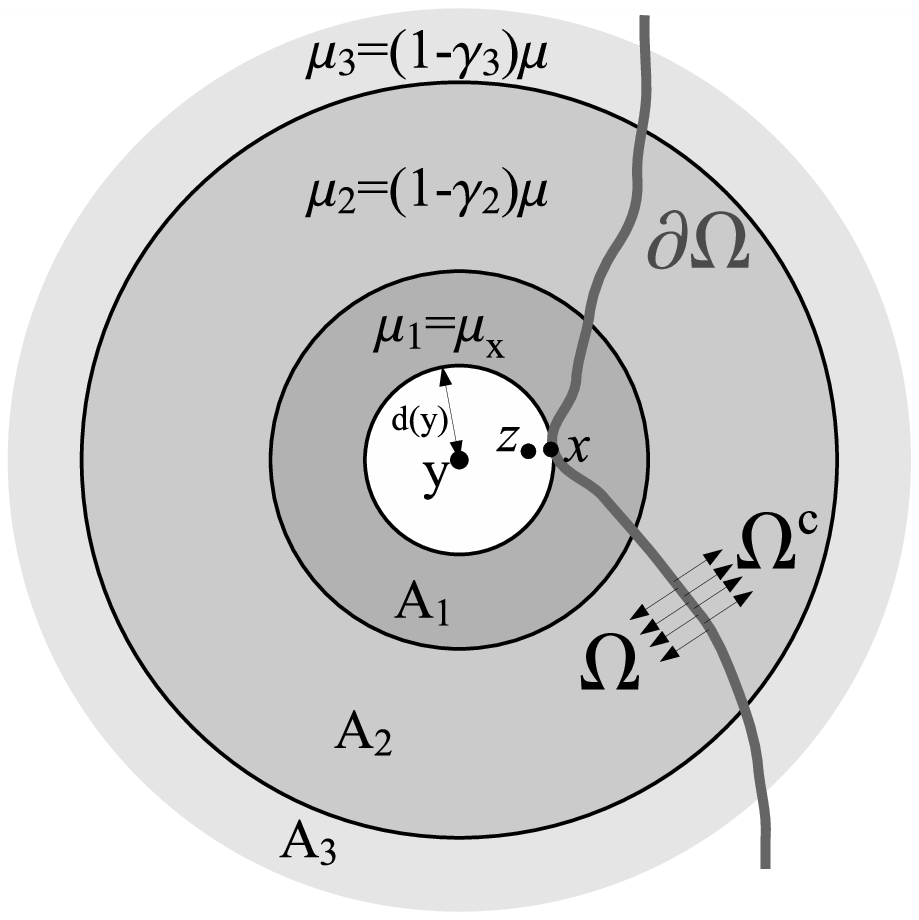}
\caption{Construction of the amalgamation $\nu=\sum_k \mu_k$}
\label{fig:ubfigA}
\end{center}
\end{figure}

Let measure $\mu_1$ be the measure $\mu_x$ from Lemma \ref{lem:cthick} restricted to $B(y, 2 d(y))$, $\mu_2=\mu_3=0$, and for $k\geq4$, let $\mu_k$ be the measure $\mu$ restricted to the $d$-dimensional annulus $$A_k=\{w\ :\ 2^{k-1} d(y)\leq |w-y|\leq2^{k} d(y)\}$$ scaled by the factor $1-\gamma_k:=1-2^{(4-k)(d-\alpha)}$. Let us also put $\ga_1=\ga_2=\ga_3=1$. We define $$\nu:=\sum_k\mu_k.$$ The ingredients of the construction are illustrated on Figure \ref{fig:ubfigA}.

Let us now prove that $\nu\in\mathcal M$. Consider any disk $B(w,r)$. Let $K$ be the largest number such that $B(w, r)$ intersects $A_K$. If $B(w,r)$ does not intersect $B(y, 2 d(y))$, the measure $\nu$ is no greater than $\mu$ on $B(w,r)$, and thus $\nu(B(w,r))\leq r^{d-\al}$. If $K\le 3$, $\nu(B(w,r))\le\mu_x(B(w,r))\leq r^{d-\alpha}$. For all other cases,
$r\geq 2^{K-3}d(y)$, which, by the choice of $\ga_K$, implies that $\ga_K r^{d-\alpha}\geq (2d(y))^{d-\al}$. Thus 
\begin{multline*}\nu(B(w,r))\leq\mu_x(B(y, 2d(y)))+\mu(B(w,r))-\sum_{k=1}^{K} \ga_k\mu(B(w,r)\cap A_k)\leq\\
(2d(y))^{d-\alpha}+(1-{\gamma_K})\mu(B(w,r))\leq \ga_K r^{d-\alpha}+(1-\ga_K) r^{d-\alpha}=r^{d-\alpha}.
\end{multline*}  
The second inequality follows from the fact that the sequence $\{\ga_k\}$ is non-increasing. 
We first apply the amalgamation construction to the case $\al=2$.

\noindent
{\bf Lemma \ref{lem:energy_upper_2} (Section \ref{subsec:al_2}):}
There exists
a constant $\delta$, dependent only on the thickness $c$, such that the following holds.
Let $\Omega$ be a $2$-thick domain. Let $y\in\Omega$ and $x\in\partial\Omega$ be the closest point to $y$. Then 
\begin{equation}
U(z)>U(y)+1 \text{ whenever } |z-x|<\delta \cdot d(y).
\end{equation}
\smallskip
\begin{proof}
Since $\mathcal M$ is a compact set, $U(y)=U^{\mu}(y)$ for some $\mu\in\mathcal M$.

Let $\mu_0$ be the restriction of the measure $\mu$ to $B(0,2)\setminus B(y, 2d(y))$. By \eqref{eq:rewrite} and \eqref{eq:measure}
\begin{equation}\label{eq:mu0}
U^{\mu_0}(y)\geq U^{\mu}(y)-\log 2.
\end{equation}

Let   $\nu$ be the amalgamation of $\mu$ at $y$.
Next, we will show that 
\begin{equation}\label{eq:up1}
U^{\nu}(z)\geq U^{\mu_0}(z)-C_1+c\cdot 2^{2-d} \log\frac{1}{\de}
\end{equation}
and
\begin{equation}\label{eq:up2}
U^{\mu_0}(z)\geq U^{\mu_0}(y)-C_2
\end{equation}
whenever $|z-x|<\de d(y)$ for some constants $C_1$ and $C_2$ depending only on $d$ and $c$.
These inequalities, together with \eqref{eq:mu0}, imply the statement of the lemma  whenever $\de$ is sufficiently small (namely, when 
$\log 1/\delta>2^{d-2} (C_1+C_2+1+\log 2)/c$).

To establish \eqref{eq:up1}, let us note that for any $k$ we have $$\mu(A_1)+\mu(A_2)+\dots+\mu(A_k)=\mu(B(y, 2^k d(y)))\leq (2^kd(y))^{d-2}.$$ By the Abel summation formula,
$$\sum_{k}\ga_k 2^{k(2-d)}\mu(A_k)\leq \sum_k d(y)^{d-2}(2^{d-2}\ga_{k-1}-\ga_{k})\cdot 2^{2-d}\leq 5 (d(y))^{d-2}.$$
This implies   
\begin{equation}
\frac1{d-2}\sum_{k\ge 2}\ga_k\int_{A_k}\frac1{|w-z|^{d-2}}\,d\mu_0(w)\leq  \sum_{k\ge 2}\ga_k\mu(A_k)(2^{k-2}d(y))^{2-d}\leq 5\cdot 4^{d-2}.
\end{equation}

Thus we obtain
\begin{multline*}%\label{eq:remainder}
U^\nu(z) \ge \int_{2\de d(y)}^{d(y)} \frac{\mu_x(B(z,r))}{r^{d-1}}\, dr+
\sum_{k\geq 2}\int_0^{\infty}\frac{\mu_k(B(z, r))}{r^{d-1}}\,dr \geq \\
\int_{2\de d(y)}^{d(y)} \frac{\mu_x(B(x,r-\de d(y)))}{r^{d-1}}\, dr+
\int_{0}^{\infty}\frac{\mu_0(B(z, r))}{r^{d-1}}\,dr-\frac1{d-2}\sum_{k\ge 1}\ga_k\int_{A_k}\frac{d\mu_0(w)}{|w-z|^{d-2}}\geq\\ \
\int_{2\de d(y)}^{d(y)} c \cdot \left(\frac{r-\de d(y)}{r}\right)^{d-2} \, \frac{dr}{r}+
\int_{0}^{\infty}\frac{\mu_0 (B(z,r))}{r^{d-1}}\,dr-5\cdot 4^{d-2}\ge c\cdot 2^{2-d} \log\frac{1}{2\de} + U^{\mu_0}(z)-5\cdot 4^{d-2}.
\end{multline*}
which implies  \eqref{eq:up1}.

To obtain \eqref{eq:up2}, note that for any point $w\in[y,z]$, for $d>2$ we have, by the estimate \eqref{eq:measure}
\begin{multline}
\left|\nabla U^{\mu_0}(w)\right|
\le\frac1{d-2}\int \left|\nabla_w\frac1{|\xi-w|^{d-2}}\right|\,d\mu_0(\xi)
= 
\int \frac1{|\xi-w|^{d-1}} \,d\mu_0(\xi) = \\ (d-1) \int_{0}^\infty \frac{\mu_0(B(w,r))}{r^d}\,dr
=(d-1) \int_{d(y)}^\infty \frac{\mu_0(B(w,r))}{r^d}\,dr \le (d-1) \int_{d(y)}^\infty \frac{r^{d-2}}{r^d}\,dr 
\leq \frac A{d(y)}
\end{multline}
for some constant $A$ depending only on $d$ and $c$. The same inequality is derived similarly in the case $d=2$. 
This implies that
\begin{equation*}
U^{\mu_0}(z)-U^{\mu_0}(y)= \int_{[y,z]}\nabla U^{\mu_0}(w)\cdot\,dw  \geq 
-|z-y|\frac A{d(y)}\geq -A,
\end{equation*}
which is exactly the equation \eqref{eq:up2}.
\end{proof}

Another application of the amalgamation construction will establish the lower bounds required in the case $\alpha<2$.

\noindent
{\bf Lemma \ref{lem:energy_upper_small} (Section \ref{subsec:al_small}):}
For any $\al<2$ and $c>0$, there exist
two constants $\delta$ and $\eta$, such that the following holds.

Let $\Omega$ be an $\alpha$-thick domain in $\RR^d$ with thickness $c$. Let $y\in\Omega$ and $x\in\partial\Omega$ be the closest point to $y$. Let $\mu\in\mathcal M$.

Then either  
\begin{equation}
U(z)>U_{\al}^{\mu}(z)+1 \text{ whenever } \delta/4\cdot d(y)<|z-x|<\delta \cdot d(y).
\end{equation}
or
\begin{equation}
\mu(B(y,2d(y)))\geq\eta d(y)^{d-\alpha}
\end{equation}
\smallskip

\begin{proof}
Let $\eta=\delta^{d-\al+2}$
Assume that $\mu(B(y, 2 d(y)))<\eta d(y)^{d-\alpha}$. Let $\mu_0$ be the restriction of $\mu$ to  $B(0,2)\setminus B(y, 2d(y))$, as in the previous lemma.  We have, by \eqref{eq:rewrite},
\begin{equation}\label{eq:up3}
U_{\alpha}^{\mu_0}(z)\geq  U_{\alpha}^{\mu}(z)-\int_{d(z)}^{2d(y)} \frac{\mu(B(z,r))}{r^{d-\al+1}} dr
\geq U_{\al}^{\mu}(z)-2d(y)\delta.
\end{equation}

On the other hand, the same reasoning as the proof of \eqref{eq:up1} above, gives
\begin{equation}\label{eq:up4}
U(z)\geq U_{\alpha}^{\nu}(z)\geq U_{\al}^{\mu_0}(z)-C_1+c\cdot 2^{\alpha-d} \log\frac{1}{\de}
\end{equation}
for some constant $C_1$ depending only on $d$, $\alpha$, and $c$. Here, as in \eqref{eq:up1}, $\nu$ is the amalgamation of $\mu$ at $y$.

Estimates \eqref{eq:up3} and \eqref{eq:up4} together imply the statement of the lemma.
\end{proof}

\subsection{The boundary drift of the WoS process: $\al\leq2$}\label{sec:drift}
First we establish that the process $U_t$ has the drift toward the boundary in the case $\alpha<2$. 

\noindent
{\bf Lemma \ref{lem:upward} (Section \ref{subsec:al_small}):}
There are  constants $L$ and $k$, depending only on $c$, $\be$, and $\alpha$,  such that $$\EE[(U_{t+k}-U_t)|U_t]>L.$$
\smallskip
\begin{proof}
Let us fix $X_t$. By weak-$^*$-compactness of the set $\mathcal M$, there exists a measure $\mu$ such that $U_{\al}^{\mu}(X_t)=U(X_t)=U_t$. By Lemma \ref{lem:energy_upper_small},  
either
\begin{equation}
U(z)>U_{\al}^{\mu}(z)+1 \text{ whenever } \delta/2\cdot d(X_t)<|z-x|<\delta \cdot d(X_t).
\end{equation}
where $x$ is the closest to $X_t$ point on $\partial\Omega$,
or
\begin{equation}
\mu(B(X_t,2d(X_t)))\geq\eta d(X_t)^{d-\alpha}
\end{equation}

Let us start with the first case.

For some $p>0$ dependent only on $d$ and $\be$, $$\left(1-\beta/2\right)^{-k}d(X_t)/2<\PP[|X_{t+k}-x|<\left(1-\beta/2\right)^{-k}d(X_t)]>p^k.$$  
Hence, for sufficiently large $k$, 
\begin{equation}\label{eq:prob}
\PP[\delta/2\cdot d(X_t)<|X_{t+k}-x|<\delta \cdot d(X_t)]>p^k.
\end{equation}

Let us now observe that by subharmonicity of the functions $U$ and $U^{\mu}_\al$, the previous estimate, the fact that 
$U\ge U^\mu_\al$ and the assumption \eqref{eq:up3},
\begin{multline}
\EE[(U_{t+k}-U_t)|X_t]= \EE[(U_{\alpha}^{\mu}(X_{t+k})-U_{\alpha}^{\mu}(X_t))|X_t]+\EE[U(X_{t+k})-U_{\alpha}^{\mu}(X_{t+k})| X_t]\geq\\ \EE[U(X_{t+k})-U_{\alpha}^{\mu}(X_{t+k})| X_t\text{ and } \delta/2\cdot d(X_t)<|X_{t+k}-x|<\delta \cdot d(X_t)]>p^k.
\end{multline}
Since the value of $X_t$ determines the value of $U_t$, this establishes the statement of Lemma in the first case (with $L=p^k$).

Now let us consider the second case. By the Green formula, for a $C^2$-smooth function $u$,
\begin{equation}\label{eq:delta_energy}
\EE[u(X_{t+1})| X_t]-u(X_t)=\int_{\beta d(X_t) S^d}u(y)\,dS(y)-u(X_t)=\int_0^{\beta d(X_t)} r^{1-d}\int_{B(X_t, r)}\Delta u(y)\,dV(y)\,dr
\end{equation}
where $S^d$ is the unit sphere in $\RR^d$ with the normalized Lebesgue measure $S$, and $dV$ is the volume element in $\RR^d$.

Note that by \eqref{eq:rewrite} and \eqref{eq:laplace}, for  $|y-X_t|\leq \be d(X_t)$ we have
\begin{multline}\label{eq:laplace_upper}
\Delta U_\alpha^{\mu}(y)=(d-\alpha+2)(2-\alpha)U_{\alpha-2}^{\mu}(y)=\\(d-\alpha+2)(2-\alpha)\int_0^{\infty}\frac{\mu(B(y, r))}{r^{d-\alpha+3}}\,dr\geq\\
(d-\alpha+2)(2-\alpha){\mu(B(X_t, 2d(X_t)))}\int_{(2+\beta)d(X_t)}^{\infty}\frac1{r^{d-\alpha+3}}\,dr\geq\\\eta(2-\alpha)\frac{\mu(B(X_t,2 d(X_t)))}{{((2+\beta)d(X_t))}^{d-\alpha+2}}\geq C_1 (d(X_t))^{-2}
\end{multline} 
for some constant $C_1$ depending only on $d$, $\alpha$, and $\beta$.

So, using \eqref{eq:delta_energy}, applied to $u=U^{\mu}$ and \eqref{eq:laplace_upper}, we get
\begin{multline}\label{eq:fast_final}
\EE[U_{t+k}|X_t]-U_t\geq \EE[U_{t+1}|X_t]-U_t\geq\EE[U^{\mu}(X_{t+1})|X_t]-U^{\mu}(X_t)\geq\\ \int_0^{\beta d(X_t)} r^{1-d}\bigl(C_2 (d(X_t))^{-2} r^d\bigr)\,dr=C_2 \left((d(X_t))^{2}\right)(d(X_t))^{-2}\beta^2/2=L
\end{multline}
for some constants $C_2$ and $L$ depending only on $d$, $\alpha$, and $\beta$. Since again the value of $X_t$ determines the value of $U_t$, the Lemma follows.
\end{proof}

Let us now turn to the case $\alpha=2$.

\noindent
{\bf Lemma \ref{lem:main} (Section \ref{subsec:al_2}):}
Let $\Omega$ be a $2$-thick domain in $\RR^d$.
There are  constants $k$ and $L$, depending only on the thickness $c$, the jump ratio $\be$, and the dimension $d$, such that $$\EE[(U_{t+k}-U_t)^2|U_t]>L.$$
\smallskip
\begin{proof}
Fix $X_t$. By Lemma \ref{lem:energy_upper_2}, there exists a constant $\delta$, dependent only on $d$ and $c$,  such that 
\begin{equation}
U(y)>U(X_t)+1 \text{ whenever } |y-x|<\delta \cdot d(X_t).
\end{equation}
This implies that $\|U_{t+k}-U_t\|^2>1$ whenever $|X_{t+k}-x|<\de d(X_t)$. Note that for some $p>0$ dependent only on $d$ and $\be$, $$\PP[|X_{t+k}-x|<\left(1-\beta/2\right)^{-k}d(X_t)]>p^k.$$  
Hence, for sufficiently large $k$, $\PP[|X_{t+k}-x|<\de d(X_t)]>p^k$, which, in turn, implies the statement of the lemma.

\end{proof}

\section{Lower bounds: examples}\label{sec:lower}

In this section we construct examples of $\al$-thick domains for which the bounds in Theorem \ref{thm:main}
are tight. The main idea of the construction is as follows. We take a domain $A$ in $\RR^d$, such as the unit ball or  a cylinder. 
We remove a ``thin" subset of points $C$ from $A$ to obtain $\Om=A\setminus C$. The set $C$ can be thought of
 as the subset of the grid $(\ga\ZZ)^d$, for some small $\ga>0$. The set $C$ will be chosen so that it ``separates'' the origin  from the boundary of $A$.
 %%, that is trapped in the annulus $B=B(0,2/3)\setminus B(0,1/3)$. 
We set $n=1/\ve$. We choose $\ga$ so that the probability of the WoS originated at $0$ hitting a $1/n$-neighborhood of 
$C$ before hitting the boundary of $A$ is $<1/2$ (this means that $C$ is ``thin"). Hence, with high probability, the WoS
will reach $\partial A$ before terminating. However, in this case the WoS will have to ``pass through'' the set $C$, 
where its step magnitudes are bounded by $\ga$. This will, in turn, yield an $\Om(1/\ga^2)$ bound on the convergence 
time. The analysis is more intricate in the case when $\al=2$. In the case when $\al>2$ is not an integer, a slight
modification to this construction is needed, as will be described below.

\subsection{Proof of the lower bound in the case $\alpha>2$}
\label{sec:lb1}

In this section we will  give an example of a ``thin'' $\alpha$-thick domain $\Om_{\alpha}$ for which the WoS will likely take $\Om(n^{2-4/{\alpha}})$ steps to converge within $\ve=1/n$ from the boundary $\Dom_{\alpha}$. The domain $\Om_\al$ will reside in $\RR^d$, where $d=\lceil \al \rceil \ge 3$. 
It is easy to see that the examples in higher dimensions $d'>d$ can be constructed from $\Om_\al$ by simply multiplying $\Om_\al$ by$[-1,1]^{d'-d}$. 

The set $$\Om_\al := \Bigl(B(0,1)_{d-1} \times [-1,1] \Bigr)\setminus S$$ is comprised of a $d$-dimensional cylinder with a set of points $S$ removed.
Here $B(0,1)_{d-1}$ denotes the unit ball in $\RR^{d-1}$. We take $A$ to be the ``middle $1/3$" shell of the $d$-dimensional cylinder:$$
A = \{ z \in \RR^{d-1}:~1/3<|z|<2/3\} \times \{ x \in [-1,1]:~1/3<|x|<2/3\}.
$$
Let $0<\ga \ll 1$ be the grid size that will be selected later. We consider the set $A_\ga$ of 
gridpoins in $A$. $$
A_\ga = (\ga \ZZ)^d \cap A. 
$$
Let $0\le \eta := d-\al <1$. Denote by $C_\eta$ the $\eta$-dimensional Cantor set in the interval $[0,1]$. It is obtained by removing the middle $\la$-fraction of the interval, then removing the middle $\la$-fraction of each subinterval etc. For the set $C_\eta$ to be $\eta$-dimensional, we choose $\la$ so that 
$$\eta = \frac{\log 2}{\log 2 - \log (1-\la)}. $$
In the special case when $\eta = 0$, we set $C_0 = \{0\}$. We now define the set $S$:
$$
S := A_\ga + \{0\} \times \ga C_\eta. 
$$
In other words, $S$ is obtained by attaching a $\ga$-scaled copy of $C_\eta$ to each gridpoint of $A_\ga$. 
This completes the definition of the set $\Om_\al = \bigl( B(0,1)_{d-1} \times [-1,1] \bigr) \setminus S$. 
Each point in $\partial\Om_\al$ has an $\eta$-dimensional Cantor set in $\Om_\al^c$ attached to 
it is captured by the following claim. Thus there is a universal constant $C\ge 1/16$ such that for every $\ga$, the set $\Om_\al$ is $\al$-thick with the thickness $C$.

The following two claims assert that for an appropriately chosen $\ga$, the WoS originated at the origin $0\in \RR^d$ and 
terminated at the $1/n$ neighborhood of $\partial\Om_\al$ is likely to hit the boundary of the external
cylinder (as opposed to the neighborhood of $S$), and is likely to spend $\Om(n^{2-4/\al})$ steps getting there. 

\begin{claim}
\label{cl:lb2}
If $\ga > 8 n^{2/\al-1}$ then a WoS originated at $0$ and terminated at the $1/n$-neighborhood of the boundary $\partial\Om_\al$ will 
hit the boundary  of the cylinder $B(0,1)_{d-1} \times [-1,1]$ with probability at least $3/4$. 
\end{claim}

\begin{proof}
It is not hard to see that we can choose a finite subset $P$ of points in $S$ such that $|P|<2 \ga^{-\al} \cdot n^{\be}$, and for every $x$ such that 
$d(x,S)<1/n$ there is a $p\in P$ such that $|x-p|<2/n$. Consider the harmonic function
\begin{equation}
\Phi(x) := \sum_{y\in P} \frac{1}{|x-y|^{d-2}}>0. 
\end{equation}
Since the function $\Phi$ is harmonic, its application to the WoS process $X_t$ gives a martingale. 
Hence if $T$ is the stopping time of the process, 
$$
\EE[\Phi(X_T)] = \Phi(X_0)=\Phi(0) < 3^{d-2} \cdot |P| < 6 \ga^{-\al}\cdot n^{\eta}.
$$ 
On the other hand, if $d(X_T,S)<1/n$, then there is a $y\in P$ with $|X_T-y|<2/n$, and 
$$
\Phi(X_T)> = 1/|y-X_T|^{d-2} >  (n/2)^{d-2}. 
$$
Hence the probability of $X_T$ being near $S$ is bounded by 
$$
\frac{\EE[\Phi(X_T)]}{(n/2)^{d-2}} < \frac{6 \ga^{-\al}\cdot n^{\eta}}{(n/2)^{d-2}}<\frac{2^{d+1}\ga^{-\al}}{n^{\al-2}}
< \frac{8^{\al} n^2}{4 (\ga n)^\al} < 1/4.
$$
The last inequality follows from the condition on $\ga$. 
\end{proof}

\begin{claim}
\label{cl:lb3}
There is a universal constant $\de>0$ such that for $\ga$ as above, with probability at least 
$1/2$ the WoS takes at least $\de (1/\ga)^2$ steps to reach the boundary of the cylinder
$B(0,1)_{d-1} \times [-1,1]$.
\end{claim}

\begin{proof}
The proof is done analogously to the proof of Claim \ref{cl:cl1} below. 
\end{proof}

Hence the expected number of steps is at least
$$
\frac{\de}{2}\cdot \left(\frac{1}{8 n^{2/\al-1}}\right)^2=\Om(n^{2-4/\al}),
$$
which completes the proof of the lower bound for Theorem \ref{thm:main} in the case when $\al>2$. 
%Proofs of Claims \ref{cl:lb1}, \ref{cl:lb2} and \ref{cl:lb3} are deferred to the Appendix. 

\subsection{Proof of the lower bound in the case $\alpha=2$}
\label{sec:lb2}

We will now give an example of a two dimensional domain $\Om$ such 
that the expected convergence time of the WoS to a $O(1/n)$-neighborhood 
of $\partial \Om$ is $\Om(\log^2 n)$.
 By taking the $d$-dimensional domain $\Om_d = \Om \times \RR^{d-2}$
for $d>2$, we obtain a lower bound of $\Om(\log^2 n)$ for $2$-thick domains 
in $\RR^d$, proving the lower bound for $\al=2$ in  Theorem \ref{thm:main}. 

The domain $\Om$ will consist of the unit disc in $\RR^2$ with $O(\log n)$ holes
``poked" out of it in a grid formation. More specifically, let $\ga=4/\log^{1/2} n$. We consider the grid 
$\Ga=\ga\ZZ \times \ga\ZZ \subset \RR^2$. We take $\Om$ to be the unit disc with points from 
$\Ga$ removed from the ``middle third" annulus of the disc. 
\begin{equation}
\label{Om2def}
\Om = B(0,1) \setminus \bigl(\left(B(0,2/3)\setminus B(0,1/3)\right) \cap \Ga\bigr). 
\end{equation}
The set $\Om$ is illustrated on Fig. \ref{fig:Om2}(a). 

We will show that a WoS originated at the origin $X_0=0$ would require an expected time of $\Om(\log^2 n)$
to converge. It is immediate to see that the same lower bound holds for any point $X_0 \in B(0,1/3)$. We first 
observe the following:

\begin{claim}
\label{cl:lb4}
 With probability at least $7/8$, a WoS originated at $X_0=0$ that runs until 
$d(X_t)<1/n$  terminates near the unit circle (and not near one of the holes). 
\end{claim}

\begin{proof}
Let 
$\{ a_i \}_{i=1}^k =B(0,1)\setminus \Om$ be the set of holes in $\Om$. Define the harmonic function 
$$
\Phi(z) = \sum_{i=1}^k \log(2/|z-a_i|). 
$$
It is clear the $\Phi(z)>0$ for all $z\in B(0,1)$. 
For any point $u$ in the $1/n$-neighborhood of any of the holes, $\Phi(u)>\log n$. On the other 
hand,
$$
\Phi(0) < k \cdot \log 6 < 2/\ga^2 = (\log n)/8. 
$$
If $X_t$ is the WoS process with $X_0=0$ terminated at time $T$ when $d(X_T,\partial\Om)<1/n$, then $\Phi(X_t)$
is a martingale. Hence,
$$
(\log n)/8 > \Phi(X_0) = \EE[ \Phi(X_t)] > \PP[\text{$X_t$ near a hole}]\cdot \log n. 
$$
Hence the probability that the WoS terminates near a hole is less than $1/8$. 
\end{proof}

 For simplicity, we 
will assume that at every step fo the process the WoS jumps exactly half way to the boundary $\partial \Om$.

To facilitate the analysis we replace the WoS process $X_t$ on $\Om$ with the following process $Y_t$. 
It evolves in exactly the same fashion as $X_t$, except when $Y_t$ is closer than $1/n$ to one of the holes 
in $\Om$. In this case, instead of terminating, the process makes a jump of $1/n$ in a direction selected
uniformly at random. The process $Y_t$ is guaranteed to terminate near the unit circle. We denote the termination 
time by $T$. Further, we set $Y_t=Y_T$ for $t>T$. Note that if the process $X_t$ does not terminate near one 
of the holes, then the process $Y_t$ coincides with $X_t$. Claim \ref{cl:lb4} implies that this happens with probability 
at least $7/8$:

\begin{claim}
\label{cl:cl0}
$\PP[\text{$X_t$ does not coinside with $Y_t$}]<1/8$.
\end{claim}

\begin{figure}[ht] 
\begin{center} \includegraphics[angle=0,scale=.75]{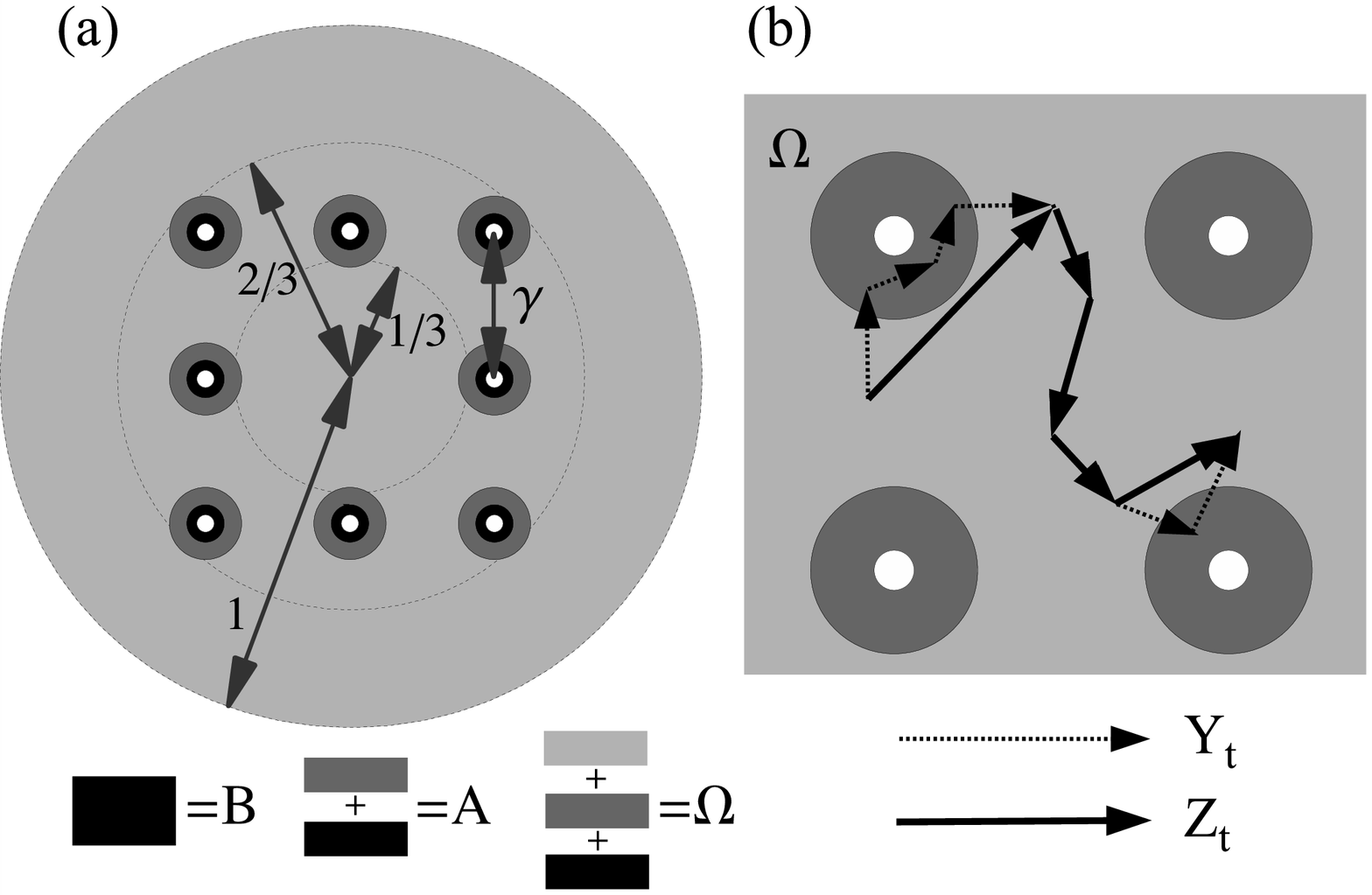}
\caption{An illustration the sets $\Om$, $A$ and $B$ (a), and a possible 
sequence of jumps in the processes $\{Y_t\}$ and $\{Z_t\}$ (b)}
\label{fig:Om2}
\end{center}
\end{figure}

We define two regions $A$ and $B$, $B \subset A\subset\Om$. We take $A$ to be the union of discs with radius $r = \ga/4$ 
around the holes in $\Om$. We take $B$ to be the union of discs with radius $r/2$ around the same holes. 
The sets $\Om$, $A$ and $B$ are illustrated on Fig. \ref{fig:Om2}(a). 

 Let time $t_0$ be the first time with $|Y_t|>1/2$. Let $t'$ be the first time afterward with either
 $|Y_t|>2/3$ or $|Y_t|<1/3$. Our goal is to show that with probability at least $3/4$, $|t_0-t'|=\Om(\log^2 n)$. 
 We define a subprocess $Z_t$ of $Y_t$ as follows. Let $\{s_i\}_{i=0}^k$ be a subsequence of times $s$
 between $t_0$ and $t'$ such that $Y_s \notin A$. We set $Z_i = Y_{s_i}$. We further define $\De_i = Z_{i}-Z_{i-1}$. 
  An instance of the process $Z_i$ is illustrated on Fig. \ref{fig:Om2}(b).
  Since $Y_t$ is a martingale, and $Z_i$ is defined by a stopping rule on $Y_t$, $Z_i$ is also a martingale, and 
  \begin{equation}
  \label{eq:mart1}
  \EE[\De_i~|~\De_1, \De_2, \ldots, \De_{i-1}]=0. 
  \end{equation}
     In addition, it is not hard to see from the definition of $Y_t$ that $|\De_i|<4/\log^{1/2}n$ for all $i$. 
     Our first claim is that the number $k$ of steps $Z_i$ is $\Om(\log n)$.
     
   \begin{claim}
   \label{cl:cl1}
   $ \PP[k<10^{-4} \log n] < 1/8$.
   \end{claim} 
    
   \begin{proof}
   Denote $\ell = 10^{-4} \log n$. Then, by \eqref{eq:mart1},
   \begin{multline*}
   \EE[(Z_0-Z_\ell)^2] = \EE[(\De_1+\De_2+\ldots+\De_\ell)^2] = \sum_{j=1}^\ell \EE[\De_j^2] + 
   \sum_{1\le i<j\le \ell} \EE[\De_i \De_j] = \\
   \sum_{j=1}^\ell \EE[\De_j^2] + 
   \sum_{1\le i<j\le \ell} \EE[\De_i \cdot\EE[\De_j|\De_i]] = 
    \sum_{j=1}^\ell \EE[\De_j^2]  < \ell \cdot 16/ \log n< 1/288.
   \end{multline*}
   On the other hand, by definition, $|Z_0-Z_k|>1/6$, and $(Z_0-Z_k)^2>1/36$. Thus, 
   $$
   \PP[k\le \ell] = \PP[Z_\ell=Z_k] < (1/288)/(1/36) = 1/8.
   $$   \end{proof}

Thus the number of steps the process $Z_t$ takes is at least $10^{-4} \log n$ w.p. $>7/8$. The
process $Y_t$ consists of the steps of the process $Z_t$ plus, in addition, steps the process takes 
within the region $A$. We claim that once the process $Y_t$ enters the region $A$, it is expected 
to spend $\Om(\log n)$ steps there. Moreover, the following holds.

\begin{claim}
\label{cl:cl2}
Let $\eta>2$. Then there is a $\theta>0$ such that whenever $Y_t \in A$, if $s>t$ is the first time, 
conditioned on $Y_t$ such that $Y_s\notin A$, then 
\begin{equation}
\PP[s-t>\theta \log^2 n] > \eta/\log n,
\end{equation}
for sufficiently large $n$. 
\end{claim}

\begin{proof}
Denote the hole in $\Om$ that is closest to $Y_t$ by $x$. 
Given that $Y_{t}\in A$, there is some fixed probability $p>0$ that $Y_{t+1}\in B$. 
In other words, $|Y_{t+1}-x|<r/2=\ga/8$. Consider the harmonic function 
$$
\Phi(z) = \log (r/|x-z|). 
$$
Let $t'>t+1$ be the first time such that either $Y_{t'}\notin A$ (and thus $t'=s$), or 
$|Y_{t'} - x|< n^{-p/(5\eta)}$. If $Y_{t'}\notin A$, then $\Phi(Y_{t'})<0$. In the other case, 
$\Phi(Y_{t'})< p \log n / (4 \eta)$. Since $t'$ is a stopping time, the optional stopping time theorem applied to the martingale $\Phi(Y_{t+\tau})$ combined with the estimate $\Phi(Y_t)>1/2$, 
gives
$$
\PP[|Y_{t'} - x|< n^{-p/(5 \eta)}] > (1/2)/(p \log n / (4 \eta)) > 2 \eta/(p \log n).
$$
To complete the argument, we claim that assuming $|Y_{t'} - x|< n^{-p/(5 \al)}$, it will take 
the process another $\Om(\log^2 n)$ steps to escape $A$ with probability at least $1/2$. We consider
the process $\phi_\tau = \Phi(Y_{t'+\tau})$ stopped at time $\tau_0$ when either $Y_{t'+\tau_0}$ escapes
$A$, or gets closer than distance $1/n$ from $x$. $\phi_\tau$ is a martingale.
Moreover, it is not hard to see that $|\phi_0 - \phi_{\tau_0}| > p \log n / (6\eta)$, and $|\phi_i - \phi_{i+1}|<1$ for 
all $i$. These two facts imply that 
$$
\EE[\tau_0] > \sum_{i=1}^{\tau_0}\EE[(\phi_i-\phi_{i-1})^2]=\EE[(\phi_{\tau_0}-\phi_0)^2]>(p \log n / (6 \eta))^2 = p^2 \log^2 n / (36 \eta^2). 
$$
Tschebyshev inequality implies   
that $\theta = p^2/(72 \eta^2)$ satisfies the statement of the claim. 
\end{proof}

By Claim \ref{cl:cl1} we know that except with probability $<1/8$ the walk will contain at least 
$\Om(\log n)$ visits to $A$. It remains to use Claim \ref{cl:cl2} to show that at least one of these
stays must be $\Om(\log^2 n)$ long. Recall that $T$ is the stopping time of the process $Y_T$, and $k$ is
the number of steps $Y_t$ takes outside of $A$. 

\begin{claim}
\label{cl:cl3}
Let $\al_1 = 10^{-4}$ from Claim \ref{cl:cl1}. There is a constant $\al_2>0$ such that 
\begin{equation}
\PP[k>\al_1 \log n \text{ and } T<\al_2 \log^2 n] < 1/8.
\end{equation}
\end{claim}
 
 \begin{proof}
For every $t$ such that $1/3<|Y_t|<2/3$ and $Y_t\notin A$, there is a probability $p_1>0$ such that 
either $Y_{t+1}\in A$ or $Y_{t+2}\in A$. By Claim \ref{cl:cl2} we can choose $\al_2>0$ such 
that whenever $Y_{t'}\in A$, the process $Y_{t'+\tau}$ does not escape $A$ for at least $\al_2 \log^2 n$
 with probability at least $p_2 = 6/(\al_1 p_1 \log n)$. Hence for each $1/3<|Y_t|<2/3$ with $Y_t\notin A$, the probability 
 that $Y_{t+1}$ or $Y_{t+2}$ enters $A$, and stays there for at least $\al_2 \log^2 n$ steps is at least $p_1\cdot p_2 = 6/(\al_1 \log n)$. 
 Since there are at least $k = \al_1 \log n$ $Y_t$'s satisfying $1/3<|Y_t|<2/3$ and $Y_t\notin A$, the probability that 
 for neither one of them does $Y_{t+1}$ or $Y_{t+2}$ enter $A$, and stay there for at least $\al_2 \log^2 n$ steps is at most 
 $$
 (1-6/(\al_2\log n))^{k/2} < (1-6/(\al_1\log n))^{\al_1 \log n/2} < e^{-(6/(\al_1\log n))\cdot(\al_1 \log n/2)}=e^{-3}<1/8.
 $$
\end{proof}

 Claims \ref{cl:cl0}, \ref{cl:cl1} and \ref{cl:cl3} imply the following. 

\begin{claim}
\label{thm:32main}
Let $X_t$ be the WoS process on the set $\Om$ with $X_0=0$. Let $T'$ be its termination time. Then 
$$
\PP[T'>\al_2 \log^2 n]>5/8,
$$
where $\al_2>0$ is the constant from Claim \ref{cl:cl3}. In particular, this implies that $\EE[T']=\Om(\log^2 n)$. 
\end{claim}

 \begin{proof}
We know that $T'>\al_2 \log^2 n$ if the following three conditions hold: (C1) the process $X_t$ coincides with 
the process $Y_t$; (C2) the process $Y_t$ makes at least $k>\al_1 \log n$ steps outside of $A$ in the $\{z:1/3<|z|<2/3\}$ annulus;
and (C3) the stopping time $T$ of $Y_t$ satisfies $T>\al_2 \log^2 n$. In fact conditions (C1) and (C3) suffice. 
We have $\PP[\overline{C1}]<1/8$ by Claim \ref{cl:cl0}, $\PP[\overline{C2}]<1/8$ by Claim \ref{cl:cl1}, and $\PP[C2\cap \overline{C3}]<1/8$ by Claim \ref{cl:cl3}. Here $\overline{C}$ denotes the complement of an event $C$. Hence 
$$
\PP[\overline{C1}\cup \overline{C2}\cup\overline{C3}] \le \PP[\overline{C1}] + \PP[\overline{C2}] + \PP[C2\cap\overline{C3}] < 3/8, 
$$
and
$$
\PP[T'>\al_2 \log^2 n]\ge \PP[C1\cap C2\cap C3] = 1- \PP[\overline{C1}\cup \overline{C2}\cup\overline{C3}] > 5/8. 
$$
\end{proof}

 Claim \ref{thm:32main} gives the lower bound for Theorem \ref{thm:main} in the case $\al=2$.

%\newpage
\bibliographystyle{alpha}
%\bibliography{biblioEuclid}
\def\cprime{$'$} \def\cprime{$'$}

%\input{appendix}
%\input{appendix2}

%\input{1intro}

%\subsection*{Acknowledgments} 
%We would like to thank Elchanan Mossel for his valuable advice on stochastic processes.
%The authors are also grateful to Stephen Cook for the numerous enlightening discussions. 

%\input{2dirich}

%\input{3derand}

%\input{4highdim}

%\newpage

\end{document}